\documentclass[11pt]{amsart}

\usepackage[utf8]{inputenc}
\usepackage{palatino}

\usepackage{amsthm,amssymb,amsfonts,amsmath,enumerate,graphicx,color}
\usepackage{mathrsfs}
\usepackage[margin=1in]{geometry}
\usepackage{setspace}

\usepackage{tikz,tikz-3dplot}

\usepackage{pgfplots}

\usepackage{mathtools}

\definecolor{cof}{RGB}{219,144,71}
\definecolor{pur}{RGB}{186,146,162}
\definecolor{greeo}{RGB}{91,173,69}
\definecolor{greet}{RGB}{52,111,72}

\tdplotsetmaincoords{70}{120}

\newcommand{\Rc}{\mathcal{R}}

\newcommand{\Nc}{\mathcal{N}}%

\newcommand{\Fc}{\mathcal{F}}
\newcommand{\Bc}{\mathcal{B}}
\newcommand{\dpt}{\operatorname{dp}}
\newcommand{\depth}{\operatorname{depth}}

\newcommand\Def[1]{\textbf{#1}}
 
\newcommand{\Z}{\mathbb{Z}}
\newcommand{\N}{\mathbb{N}}
\newcommand{\R}{\mathbb{R}}

\newcommand{\Des}{\mathrm{Des}}
\newcommand{\maj}{\mathrm{maj}}
\newcommand{\des}{\mathrm{des}}

\newcommand{\conv}{\mathrm{conv}}

\renewcommand{\phi}{\varphi}
\renewcommand{\emptyset}{\varnothing}

\def\x{{\boldsymbol x}}

\newcommand\commentout[1]{}

\newtheorem{theorem}{Theorem}[section]
\newtheorem{corollary}[theorem]{Corollary}
\newtheorem{proposition}[theorem]{Proposition}

\theoremstyle{remark}

\newtheorem{remark}[theorem]{Remark}
\theoremstyle{definition}
\newtheorem{definition}[theorem]{Definition}

\begin{document}
\title{Multivariate polynomials for generalized permutohedra}
\author{Eric Katz}
\address{Department of Mathematics\\
         The Ohio State University\\
        Columbus, OH 43210}
\email{katz.60@osu.edu}        
\author{McCabe Olsen}
\address{Department of Mathematics\\
         The Ohio State University\\
        Columbus, OH 43210}
\email{olsen.149@osu.edu}
\date{\today}

\begin{abstract}
Using the notion of Mahonian statistic on acyclic posets, we introduce a $q$-analogue of the $h$-polynomial of a simple generalized permutohedron. 
We focus primarily on the case of nestohedra and on explicit computations for many interesting examples, such as $S_n$-invariant nestohedra, graph associahedra, and Stanley--Pitman polytopes.
For the usual (Stasheff) associahedron, our generalization yields an alternative $q$-analogue to the well-studied Narayana numbers.  
\end{abstract}

\thanks{The authors thank Vic Reiner for helpful comments. 
The first author was partially supported by NSF DMS 1748837.
}




\maketitle


\section{Introduction}

Given any combinatorially defined polynomial, a common theme in enumerative combinatorics is to consider multivariate analogues which further stratify and enrich the encoded data by an additional combinatorial statistic.
A notable example of is the \emph{Euler--Mahonian polynomial}
	\[
	A_n(t,q)=\sum_{\pi\in S_n}t^{\des(\pi)}q^{\maj(\pi)}
	\]
which is a bivariate generalization of the more foundational \emph{Eulerian polynomial}
	\[
	A_n(t)=\sum_{\pi\in S_n}t^{\des(\pi)}.
	\]
In this case, the we further startify the descent statistic on permutations by the additional data of the major index.
Such a generalization is commonly referred to as a $q$-analogue in reference to usual choice of added variable.

Given a convex polytope $P\subset\R^n$, the $h$-polynomial is an encoding of the face numbers of $P$ obtained as a linear change of variables of the generating function for the face numbers. If $P$ is simple or simplicial, then the Dehn--Sommerville equations for $P$ are reflected in the palindromicity of the $h$-polynomial. For simple rational polytopes, the $h$-polynomial is the Poincar\'{e} polynomial of the cohomology groups of the toric variety attached to the polytope. Moreover, for simplicial polytopes, the $h$-polynomial is the generating function for facets of $P$ according to the size of their restriction sets  {\cite[Sec. 8.3]{Ziegler-Book}.

Generalized permutohedra are a broad class of convex polytopes which exhibit many nice properties.
First introduced by Postnikov \cite{Postnikov-Beyond}, these polytopes have been the subject of much study and are of wide interest in many areas of algebraic and enumerative combinatorics, including the combinatorics of Coxeter groups, cluster algebras, combinatorial Hopf algebras and monoids, and polyhedral geometry (see, e.g., \cite{AguiarArdila,Armstrong,FeichtnerSturmfels,FominReading}).

Of particular interest for our purposes, Postikov, Reiner, and Williams \cite{PostnikovReinerWilliams} give a combinatorial description of the $h$-polynomial for any simple generalized permutohedron using an Eulerian descent statistic on posets. 
Moreover, they provide a formula for well-behaved, special cases of generalized permutohedra.
We give a bivariate generalization of their description for any simple generalized permutohedron:
for $P$ be a simple generalized permutahedron and $Q_\sigma$ the poset for a full dimensional cone $\sigma$ in the normal fan $\Nc(P)$, we define
	\[
	h_P(t,q)\coloneqq \sum_{\sigma\in\Nc(P)}t^{\des(Q_\sigma)}q^{\maj(Q_\sigma)}
	\]
where $\des$ and $\maj$ are statistics defined below.
Furthermore, we are able to be more explicit when restricting to particular classes of generalized permutahedra, specifically $S_n$-invariant nestohedra, graph associahedra, and Stanley--Pitman polytopes.
 
Our definition of the bivariate $h$-polynomial, which specializes to the usual $h$-polynomial is justified by analogy with the Euler--Mahonian polynomial. Other possible definitions exist. An inequivalent definition is the principal specialization of the Frobenius characteristic of the permutohedral toric variety. This definition does not extend to generalized permutohedra and is not discussed in the body of the paper. However, it does make use of the major index.

The structure of this note is as follows. 
In Section~\ref{sec:background}, we provide a review of necessary background and terminology on permutations, posets, polyhedral geometry, and generalized permutohedra. 
Section~\ref{sec:general} defines and discusses the the $q$-analogue for the $h$-polynomial of any simple generalized permutohedron.
In Section~\ref{sec:Nestohedra}, we focus on general results for a large class of simple generalized permutohedra called \emph{nestohedra}, including a palidromicity result for special cases. 
Section~\ref{sec:Examples} is devoted to several explict examples, including $S_n$-invariant nestohedra, graph associahedra, the classical associahedron, the stellohedron, and the Stanley--Pitman polytope. 
These examples produce some \emph{alternative} $q$-analogues of some well-known combinatorial sequences, including the Narayana numbers. 

\section{Background}
\label{sec:background}
In this section, we provide a brief review of basic properties of permutations statistics, posets, polytopes and normal fans, and generalized permutohedra.

\subsection{Permutation statistics}
Let $A=\{a_1<a_2<\ldots<a_n\}$ be a set of $n$ elements.
The \Def{symmetric group on $A$}, denoted $S_A$, is the set of all permutations of the elements of $A$.
In the case of $A=[n]$, we will simply write $S_n$.
Given $\pi=\pi_1\pi_2\cdots\pi_n\in S_A$, the \Def{descent set} of $\pi$ is 
	\[
	\Des(\pi)=\{i\in[n-1] \ : \ \pi_i>\pi_{i+1}\},
	\] 	
the \Def{descent number} of $\pi$ is $\des(\pi)=|\Des(\pi)|$, and the \Def{major index} of $\pi$ is 
	\[
	\maj(\pi)=\sum_{i\in\Des(\pi)}i.
	\]

The descent statistic is commonly referred to as an \emph{Eulerian} statistic, due to the connection to polynomial first studied by Euler \cite{Euler}.
The \Def{Eulerian polynomial}  $A_n(t)$ is the unique polynomial which satisfies 
	\[
	\sum_{k\geq 0} (k+1)^n t^k = \frac{A_n(t)}{(1-t)^{n+1}}
	\]
However, this polynomial can be interpreted entirely combinatorially as
	\[
	A_n(t)=\sum_{\pi\in S_n}t^{\des(\pi)}.
	\]	
The major index, on the other hand, is commonly known as a \emph{Mahonian statistic}, as it was introduced by MacMahon \cite{MacMahon}. 
The descent statistic and major index statistic are naturally linked as they both encode information regarding the descent set of a permutation. 
Thus, it is fruitful to consider the joint distribution of these statistics, which motivates the \Def{Euler--Mahonian polynomial}
	\[
	A_n(t,q)=\sum_{\pi\in S_n} t^{\des(\pi)} q^{\maj(\pi)},
	\]
which specializes to the Eulerian polynomial under the substitution $q=1$.	
This polynomial and various generalizations are widely of interest (see, e.g., \cite{AdinBrentiRoichman-Hyper,BagnoBiagioli,BeckBraun-EM,BraunOlsen-EulerMahonianSemigroup}). 

\subsection{Posets}
Let $Q$ be a partially ordered set (poset) on $[n]$ with relation $<_Q$. 
Given $x,y\in Q$, let $x\lessdot_Q y$ denote the covering relation.
Two elements $x,y\in Q$ are \Def{incomparable} if we have neither $x<_Q y$ nor $y<_Q x$.
A \Def{chain} in $Q$, is a collection of elements $x_1,x_2,\ldots,x_k\in Q$ such that $x_1<_Q\cdots <_Q x_k$.
A chain $x_1,x_2,\ldots,x_k\in Q$ is called \Def{saturated} if $x_1\lessdot_Q\cdots \lessdot_Q x_k$.
The \Def{Hasse diagram} of $Q$ is graph with an oriented upwards direction such that edges encode covering relations.
We say that $Q$ is \Def{acyclic} if for all $x,y\in[n]$ with  $x<_Q y$ there is a unique saturated chain from $x$ to $y$. 

Given two posets $Q_1$ and $Q_2$, the \Def{ordinal sum} $Q_1\oplus Q_2$ is the poset on the disjoint union of the ground sets of $Q_1$ and $Q_2$ such that $x<y$ if  (i) $x,y\in Q_1$ and $x<_{Q_1}y$, (ii) $x,y\in Q_2$ and $x<_{Q_2} y$, or (iii) $x\in Q_1$ and $y\in Q_2$.

The poset $Q$ is called \Def{graded} (or \Def{ranked}) if there is a function $\rho:Q\to \N$ such that $\rho(x)\geq 0$ for all $x\in Q$ and if $x\lessdot_Q y$, then $\rho(y)=\rho(x)+1$.
While there are infinitely many rank functions for a graded $Q$, there is a unique minimal rank function $\rho$ such that $\rho(x)-1$ is not a valid rank function.

Given a poset $Q$ on $[n]$, we can generalize the notion of the descent statistic for permutations.
The \Def{descent set} of $Q$ is 
	\[
	\Des(Q)\coloneqq\left\{ (i,j) \, : \, i\lessdot_{Q} j \mbox{ and } i>_{\N} j \right\}
	\]
and thus the \Def{descent number} of $Q$ is $\des(Q)\coloneqq|\Des(Q)|$.
If $Q$ is a graded poset on $[n]$ with minimal rank function $\rho$, we further have a notion of \Def{major index} of $Q$
	\[
	\maj(Q)\coloneqq \sum_{(i,j)\in\Des(Q)}\rho(j).
	\]
We note that if $Q$ is a totally ordered set with labels $\pi_1<_Q\pi_2<_Q\cdots<_Q\pi_n$, these quantities are precisely $\des(\pi)$ and $\maj(\pi)$.	

\subsection{Polytopes, fans, and $h$-vectors}
	A \Def{(convex) polytope} $P$ is the convex hull of finitely many point $\x_1,\ldots,\x_k\in \R^d$.
	The \Def{dimension} of $P$, denoted $\dim(P)$, is the dimension of the smallest affine subspace containing $P$.
	A \Def{face} $F$ of $P$ is the collection of points where a linear functional $\ell\in(\R^n)^\ast$ is maximized on $P$. 
	Faces of dimension $0$ are called \Def{vertices} and faces with $\dim(F)=\dim(P)-1$ are called \Def{facets}.
	A polytope $P$ is called \Def{simple} if every vertex is contained in exactly $\dim(P)$ many facets.
	The set of all faces of $P$ forms a poset $L(P)$ under inclusion of faces, which we will the \Def{face lattice} of $P$.
	We say that two polytopes $P_1$ and $P_2$ are \Def{combinatorially equivalent} if $L(P_1)=L(P_2)$.
	
	A \Def{polyhedral cone} $\sigma\subset \R^n$ is solution set to the weak inequality $A\x\geq 0$ for some real matrix $A$.
	A cone $\sigma$ is called \Def{pointed} if $\sigma$ contains no linear subspaces.
	The \Def{dimension} of $\sigma$, denoted $\dim(\sigma)$, is the dimension of the smallest affine subspace containing $\sigma$.
	A cone is called \Def{simplical} if it is defined by exactly $\dim(\sigma)$ many independent inequalities.
	A \Def{face} of $\sigma$ is the subset obtained by replacing some of the defining equalities with equality. 
	Two cones $\sigma_1$ and $\sigma_2$ \Def{intersect properly} if $\sigma_1\cap\sigma_2$ is a face of both $\sigma_1$ and $\sigma_2$.
	A collection of cones $\Fc$ is called a \Def{fan} if it closed under taking faces, and any two cones $\sigma_1,\sigma_2\in \Fc$ intersect properly.
	We say $\Fc$ is a \Def{complete fan} if $\Fc$ covers $\R^n$.
	
	Let $P$ be a polytope with face $F$.
	The \Def{normal cone of $F$ in $P$}, denoted $\Nc_P(F)$  is the subset of linear functions $\ell\in(\R^n)^\ast$ whose maximum on $P$ occurs for all points on $F$. That is,
	\[
	\Nc_P(F)\coloneqq \left\{ \ell\in (\R^n)^\ast \ : \ \ell(x)=\max\{\ell(y) \ : \ y\in P\} \mbox{ for all } x\in F \right\}
	\]
The \Def{normal fan} of $P$, denote $\Nc(P)$ is the complete fan formed by the normal cones of all faces.
Note that $\Nc(P)$ is pointed if and only if $\dim(P)=n$.
However, one can always reduce $\Nc(P)$ to a pointed fan in the space $(\R^n)^\ast/P^{\bot}$, where $P^{\bot}\subset(\R^n)^\ast$ is the subset of linear functionals constant on $P$.

Given a polytope $P$, the \Def{$f$-vector} of $P$ is the vector $(f_0(P),f_1(P),\ldots, f_{\dim(P)}(P))$ where $f_i(P)$ is the number of $i$-dimensional faces of $P$.
The \Def{$f$-polynomial} of $P$ is the generating function $f_P(t)=\sum_{i=0}^{\dim(P)}f_i(P)t^i$. 
Moreover, one can define $f$-vector and $f$-polynomial of fan $\Fc$ in the obvious way. 
The $f$-vectors of a polytope $P$ and its normal fan $\Nc(P)$ are related by $f_{i}(P)=f_{\dim(P)-i}(\Nc(P))$.

Given $P$ a simple polytope, or equivalently if $\Fc$ is a simplical fan, one can instead consider a different vector.
The  \Def{$h$-vector} of $P$ is the vector $(h_0(P),\ldots,h_{\dim(P)}(P))\in \Z_{\geq 0}^{\dim(P)+1}$ 
 and the \Def{$h$-polynomial} is $h_P(t)=\sum_{i=0}^{\dim(P)}h_i(P)t^i$ defined uniquely by the relation $f_P(t)=h_P(t+1)$.
 Likewise, the $h$-polynomial of $\Fc$, $h_\Fc(t)=\sum_{i=0}^{\dim(\Fc)}h_i(\Fc)t^i$ is given by the relation $t^d f_\Fc(t^{-1})=h_\Fc(t+1)$.
Hence, the $h$-polynomial of a polytope $P$ and the $h$-polynomial of its normal fan $\Nc(P)$ coincide.
In this case, it happens that the $h$-polynomial satisfies the \Def{Dehn-Sommerville relations} $h_i(P)=h_{\dim(P)-i}(P)$ for $i=0,1,\ldots, \dim(P)$ (see, e.g., {\cite[Sec. 8.3]{Ziegler-Book}}). 
		
\subsection{Generalized permutohedra}
\label{sec:genperm}
Given $\alpha=(\alpha_1,\alpha_2,\ldots,\alpha_n)\in\R^n$ such that $0\leq \alpha_1<\alpha_2<\cdots<\alpha_n$, the \Def{$\alpha$-permutohedron} or \Def{usual permutohedron} $\Pi^\alpha_n\subset\R^n$ is the convex hull of the $S_n$-orbit of $\alpha$. 
Note that this is an $(n-1)$-dimensional polytope, as it lies in the hyperplane $\sum_{i=1}^n\x_i = \sum_{j=0}^n\alpha_j$.
Regardless of the choice of $\alpha$, the normal fan of $\Pi_n^\alpha$ is the \Def{braid fan} is
\newcommand{\Brn}{\operatorname{Br}_n}	
	\[
	\Brn\coloneqq \{ \sigma(\pi) \ : \ \pi\in S_n\}\subseteq \R^n/(1,1,\ldots,1)
	\]
where the full dimensional cones $\sigma(\pi)$ are
	\[
	\sigma(\pi)=\{x\in \R^n/(1,1,\ldots,1) \ : \ x_{\pi_1}\leq x_{\pi_2}\leq \cdots \leq x_{\pi_n}\}.
	\]
See Figure~\ref{fig:braid3} for the example of $\operatorname{Br}_3$.	
Given that any choice of $\alpha$ produces the normal fan of $\Brn$, we will usually consider usual permutohedron for $\alpha=(0,1,2,\ldots,n-1)$, which we will simply denote $\Pi_n$. As a consequence of its description in terms of restriction sets, the $h$-polynomial for $\Pi_n$ is given by the Eulerian polynomial 
	\[
	h_{\Pi_n}(t)=\sum_{\pi\in S_n}t^\des(\pi).
	\]  
	
\begin{figure}
	\centering 
	\begin{tikzpicture}
	\draw[black,thick, <->] (0,2.2)--(0,-2.2);
	\draw[black,thick, <->] (-1.9,1.097)--(1.9,-1.097);
	\draw[black,thick, <->] (1.9,1.097)--(-1.9,-1.097);
	
	\node at (.75,1.2) {$123$};
	\node at (-.75,1.2) {$213$};
	\node at (.75,-1.2) {$312$};
	\node at (-.75,-1.2) {$321$};
	\node at (1.25,0) {$132$};
	\node at (-1.25,0) {$231$};
	
	\node at (0,2.3) {\tiny $x_1=x_2$};
	\node at (2.2,1.29) {\tiny $x_2=x_3$};
	\node at (-2.2, 1.29) {\tiny $x_1=x_3$};
	
	\draw[fill=red,opacity=0.3] (0,0)--(0,2)--(1.732,1)--cycle;
	\draw[fill=red,opacity=0.3] (0,0)--(0,-2)--(1.732,-1)--cycle;
	\draw[fill=red,opacity=0.3] (0,0)--(-1.732,1)--(-1.732,-1)--cycle;
	
	\draw[fill=blue,opacity=0.3] (0,0)--(0,2)--(-1.732,1)--cycle;
	\draw[fill=blue,opacity=0.3] (0,0)--(0,-2)--(-1.732,-1)--cycle;
	\draw[fill=blue,opacity=0.3] (0,0)--(1.732,1)--(1.732,-1)--cycle;
	\end{tikzpicture}
\caption{$\operatorname{Br}_3$ in $\R^3/(1,1,1)$}
\label{fig:braid3}	
\end{figure}		

Introduced by Postnikov in \cite{Postnikov-Beyond}, a \Def{generalized permutohedron} $P\subset \R^n$ is a convex polytope whose normal cone $\Nc(P)\subset \R^n/(1,1,\ldots,1)$ can be refined to $\Brn$. 
We say that $\Nc(P)$ is a \Def{coarsening} of $\Brn$ if there is a polytopal realization for $\Nc(P)$ which can be refined by $\Brn$

Suppose that $\Fc$ is a coarsening of $\Brn$. 
Given a full-dimensional cone $\sigma\in \Fc$, one can associate a poset on $[n]$ to the cone $\sigma$,  denoted $Q_{\sigma}$, by the relations $i<_{Q_\sigma}j$ if $x_{i}\leq x_j$ for all $\x\in \sigma$.
The poset $Q_\sigma$ is connected and acyclic if and only if $\sigma$ is a full-dimensional simplical cone. 
The reader should consult {\cite[Sec. 3]{PostnikovReinerWilliams}} for additional exposition and details on this correspondence. 

In the case of a simple generalized permutohedron, or rather a simplicial coarsening of $\Brn$, one can give a combinatorial formula for the $h$-polynomial in terms of descents on acyclic posets, which is a natural generalization of the result for the usual permutohedron.	
\begin{theorem}[{\cite[Thm. 4.2]{PostnikovReinerWilliams}}]
\label{thm:hpolyGeneral}
Let $P$ be a simple generalized permutahedron and let $\{Q_\sigma\}_{\sigma\in \Nc(P)}$ be the posets for full dimensional cones in the normal fan $\Nc(P)$. Then
	\[
	h_P(t)=\sum_{\sigma\in\Nc(P)}t^{\des(Q_\sigma)}.
	\]
\end{theorem}

\section{Simple generalized permutohedra}
\label{sec:general}

In this section, we will introduce a bivariate generalization for the $h$-polynomial of any simple generalized permutohedron.
Particularly, we will give a formula for a $q$-analogue of Theorem \ref{thm:hpolyGeneral}, which gives us the expected bivariate polynomial in the case of $\Pi_n$. 
Unfortunately, our generalization does not produce a polynomial invariant for the combinatorial type of $\Nc(P)$.
Rather, the polynomials will vary based upon the particular choice of coarsening of $\Brn$, and thus one may have combinatorially equivalent generalized permutohedra with different polynomials. 

Let $P\subset \R^n$ be a simple generalized permutohedron with normal fan $\Nc(P)\in \R^n/(1,1,\ldots, 1)$.
Given any full dimensional cones $\sigma\in \Nc(P)$, let $Q_{\sigma}$ denote the poset associated with $\sigma$.
First, observe the following.

\begin{proposition}
\label{prop:posetrank}
Given $\sigma\in\Nc(P)$ be a full dimensional cone for a simple generalized permutahedron $P$. Then the poset $Q_\sigma$ is an acyclic, graded poset with a unique minimal rank function $\rho:Q_\sigma\to \N$.  
\end{proposition}
\begin{proof}
Since $P$ is simple, this implies that $\sigma$ is simplicial. 
Thus, by observations in Section \ref{sec:genperm}
(c.f. {\cite[Prop. 3.5]{PostnikovReinerWilliams}}),  the poset $Q_\sigma$ is acyclic.
This implies that if $x<_{Q_\sigma} y$, there is a unique saturated chain from $x$ to $y$.
Hence, we can define a rank function $\rho$ such that $\rho(x)\geq 0$ for all $x\in Q_{\sigma}$ and if $x\lessdot_{Q_\sigma} y$ then $\rho(y)=\rho(x)+1$.
To obtain the unique minimal rank function, let $\rho$ be any valid function above and define $\tilde{\rho}(x)=\rho(x)-\alpha$, where $\alpha=\min_{y\in Q_\sigma}\rho(y)$.
\end{proof}

Based on the obserations of Proposition~\ref{prop:posetrank}, we can now give a $q$-analogue of the $h$-polynomial for a simple generalized permutohedron.

\begin{definition} \label{d:bivariatehpolynomial}
Let $P$ be a simple generalized permutahedron and let $\{Q_\sigma\}_{\sigma\in \Nc(P)}$ be the posets for full dimensional cones in the normal fan $\Nc(P)$. Then, the \Def{$q$-$h$-polynomial} is given by 
	\[
	h_P(t,q)\coloneqq \sum_{\sigma\in\Nc(P)}t^{\des(Q_\sigma)}q^{\maj(Q_\sigma)}.
	\]
\end{definition}

In the case of the usual permutohedron $\Pi_n$, this $q$-analogue gives us the expected generalization. 
The full dimensional cones of the braid fan correspond to permutations $\pi\in S_n$ giving the total order $Q_\pi$ which is $\pi_1<\pi_2<\cdots<\pi_n$.
By definition, $\des(Q_\pi)=\des(\pi)$ and $\maj(Q_\pi)=\maj(\pi)$.
Thus we have 
	\[
	h_{\Pi_n}(t,q)=\sum_{\pi\in S_n}t^{\des(\pi)}q^{\maj(\pi)}
	\] 
which is the Euler--Mahonian polynomial, an expected $q$-analogue of the Eulerian polynomial.

Unfortunately, this construction is not invariant under reordering of the ground set. 
That is, the $q$-analogue depends on the choice of embedding or (equivalently) the choice of coarsening of the braid fan.
For example, consider the associahedron $\mathsf{A}(3)\subset \R^3$ which is the polytope whose normal fan is obtained by merging exactly $2$ full-dimensional cones that intersect in an edge in $\operatorname{Br}_3$ (see Section \ref{sec:associahedron} for an in depth discussion of $\mathsf{A}(n)$). 
Two different choices of coarsening will produce combinatorially equivalent fans (resp. polytopes), but different multivariate polynomials.
If one coarsens the braid fan by merging the cones corresponding to the permutations $132$ and $312$, the obtained $q$-analogue is $h_{\mathcal{F}_1}(t,q)=1+tq+2tq^2+t^2q^3$.
Alternatively, if one instead coarsens the braid fan by merging the cones corresponding to $231$ and $321$, the obtained $q$-analogue is $h_{\mathcal{F}_2}(t,q)=1+2tq+tq^2+t^2q^2$. 
Of course when $q=1$ in either case we have $h_{\mathcal{F}}(t)=1+3t+t^2$ as expected.
These two choices of coarsening are depicted in  Figure~\ref{fig:A3choices}.

\begin{figure}
	\centering 
	\begin{tikzpicture}
	\draw[black,thick, <->] (0,2.2)--(0,-2.2);
	\draw[black,thick, <-] (-1.9,1.097)--(0,0);
	\draw[black,thick, <->] (1.9,1.097)--(-1.9,-1.097);
	
	\node at (.75,1.2) {$123$};
	\node at (-.75,1.2) {$213$};
	\node at (-.75,-1.2) {$321$};
	\node at (-1.25,0) {$231$};
	
	\node at (0,2.3) {\tiny $x_1=x_2$};
	\node at (2.2,1.29) {\tiny $x_2=x_3$};
	\node at (-2.2, 1.29) {\tiny $x_1=x_3$};
	
	\draw[fill=red,opacity=0.3] (0,0)--(0,2)--(1.732,1)--cycle;
	
	\draw[fill=red,opacity=0.3] (0,0)--(-1.732,1)--(-1.732,-1)--cycle;
	
	\draw[fill=blue,opacity=0.3] (0,0)--(0,2)--(-1.732,1)--cycle;
	\draw[fill=blue,opacity=0.3] (0,0)--(0,-2)--(-1.732,-1)--cycle;
	
	\draw[fill=green, opacity=0.3] (0,0)--(0,-2)--(1.732,1)--cycle;
	
	\draw[fill=black] (.4,-.1) circle(1.5pt);
	\draw[fill=black] (.2, -.5) circle(1.5pt);
	\draw[fill=black] (.6,-.5) circle(1.5pt);
	\draw[black] (.2,-.5)--(.4,-0.1)--(.6,-.5);
	\node at (.2,-.7) {\tiny $1$};
	\node at (.6,-.7) {\tiny $3$};
	\node at (.4, .07) {\tiny $2$};

	\end{tikzpicture}
\hspace{1cm}	
	\begin{tikzpicture}
	\draw[black,thick, <->] (0,2.2)--(0,-2.2);
	\draw[black,thick, <->] (-1.9,1.097)--(1.9,-1.097);
	\draw[black,thick, <-] (1.9,1.097)--(0,0);
	
	\node at (.75,1.2) {$123$};
	\node at (-.75,1.2) {$213$};
	\node at (.75,-1.2) {$312$};
	\node at (1.25,0) {$132$};
	
	\node at (0,2.3) {\tiny $x_1=x_2$};
	\node at (2.2,1.29) {\tiny $x_2=x_3$};
	\node at (-2.2, 1.29) {\tiny $x_1=x_3$};
	
	\draw[fill=red,opacity=0.3] (0,0)--(0,2)--(1.732,1)--cycle;
	\draw[fill=red,opacity=0.3] (0,0)--(0,-2)--(1.732,-1)--cycle;
	
	\draw[fill=blue,opacity=0.3] (0,0)--(0,2)--(-1.732,1)--cycle;
	\draw[fill=blue,opacity=0.3] (0,0)--(1.732,1)--(1.732,-1)--cycle;
	
	\draw[fill=green, opacity=0.3] (0,0)--(0,-2)--(-1.732,1)--cycle;
	
	\draw[fill=black] (-.4,-.1) circle(1.5pt);
	\draw[fill=black] (-.2, -.5) circle(1.5pt);
	\draw[fill=black] (-.6,-.5) circle(1.5pt);
	\draw[black] (-.2,-.5)--(-.4,-0.1)--(-.6,-.5);
	\node at (-.2,-.7) {\tiny $2$};
	\node at (-.6,-.7) {\tiny $3$};
	\node at (-.4, .07) {\tiny $1$};
	\end{tikzpicture}
\caption{Two coarsenings of $\operatorname{Br}_3$ which are combinatorially equivalent by produce different $q$-$h$-polynomials.}
\label{fig:A3choices}	
\end{figure}
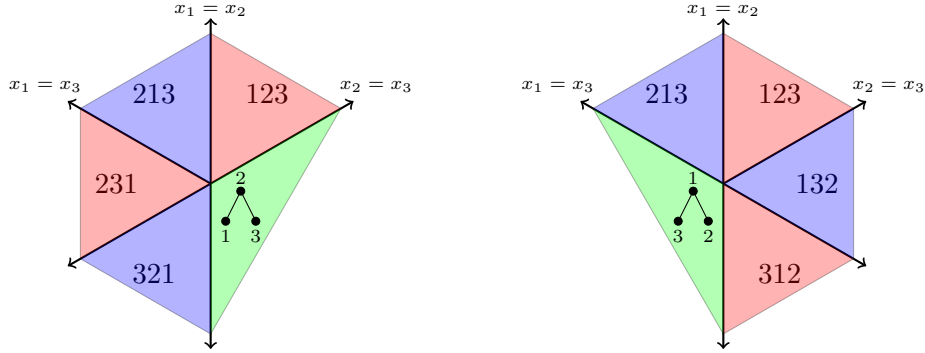

\section{Nestohedra}
\label{sec:Nestohedra}
In this section, we focus on a broad class of simple generalized permutohedra known as nestohedra, for which one can be more explicit.
The nestohedra were first introduced by Postnikov \cite{Postnikov-Beyond}. 
To construct a nestohedron, we need the notion of a building set.

\begin{definition}[{\cite[Def. 7.1]{Postnikov-Beyond}}]
A collection $\mathcal{B}$ of nonempty subsets of $[n]$ is called a \Def{building set} if it satisfies the following conditions:
	\begin{enumerate}
	\item If $I,J \in \mathcal{B}$ and $I\cap J\neq \emptyset$, then $I\cup J\in \mathcal{B}$.
	\item $\mathcal{B}$ contains all singletons $\{i\}$, such that $i\in[n]$.
	\end{enumerate}	
\end{definition}

A building set $\mathcal{B}$ is \Def{connected} if  $[n]\in \mathcal{B}$. 
For any building set, one can define a nestohedron.
For any subset $I\subseteq [n]$, let $\Delta_I\coloneqq \conv\{ e_i \ : \ i\in I\}$.

\begin{definition}[(see {\cite[Def. 6.3]{PostnikovReinerWilliams}}]
Given a building set $\mathcal{B}$ on $[n]$.
The \Def{nestohedron} $P_{\mathcal{B}}$ on the building set is the polytope obtained from the Minkowski sum
	\[
	P_{\mathcal{B}}\coloneqq \sum_{I\in \mathcal{B}}y_I\Delta_I
	\]
for some strictly positive parameters $y_I$. 	
\end{definition}

One can see that $P_{\mathcal{B}}$ is a generalized permutohedron because $\Nc(\Delta_I)$ is  refined by $\Brn$ and thus $\Nc(\sum_{I\in \mathcal{B}}y_I\Delta_I)$ must also be refined by $\Brn$  {\cite[Prop. 7.12]{Ziegler-Book}}.
For our purposes, we are primarily interested in the explicit cones and associated posets in $\Nc(P_{\mathcal{B}})$.
These can be described through combinatorial means.
Given a rooted tree $T$ on $[n]$ which is directed such that all edges are oriented away from the root and a vertex $i$ in $T$,  let $T_{\leq i}$ be the tree of descendants of $i$. That is, $j\in T_{\leq i}$ if there is a directed path from $i$ to $j$ in $T$.
We define $\mathcal{B}$-trees for a connected building set $\mathcal{B}$. 

\begin{definition}[{\cite[Def. 7.7]{Postnikov-Beyond}}]
For a connected building set $\mathcal{B}$ on $[n]$, a \Def{$\mathcal{B}$-tree} is a rooted tree $T$ on the set $[n]$ such that 
	\begin{enumerate}
	\item For any $i\in[n]$, one has $T_{\leq i}\in \mathcal{B}$
	\item For any $k\geq 2$ incomparable nodes $i_1,\ldots, i_k\in [n]$, one has $\bigcup_{j=1}^k T_{\leq i_j}\not\in \mathcal{B}$.
	\end{enumerate}
\end{definition}

One can algorithmically construct all of these $\mathcal{B}$-trees using the following proposition.

\begin{proposition}[{\cite[Prop. 8.5]{PostnikovReinerWilliams}}]
\label{prop:BTreeConstruction}
Let $\Bc$ be a connected building set on $[n]$ and let $i\in[n]$.
Let $\Bc_1,\ldots,\Bc_r$ be the connected components of of the restriction $\Bc|_{[n]\setminus \{i\}}$.
Then all $\Bc$-trees with root at $i$ are obtained by picking a $\Bc_j$-tree $T_j$, for each component $\Bc_j$, $j=1,\ldots,r$, and connecting the roots of $T_1,\ldots,T_r$ to the node $i$ by edges.
\end{proposition}

For a building set $\mathcal{B}$, a $\mathcal{B}$-tree $T$ has the structure of a poset by $x\lessdot y$ provided that the there is an edge $(x,y)$ and $y$ is closer to the root.
For ease of notation, we will write $x\lessdot_T y$ to denote an edge $(x,y)$ in $T$ and to indicate which element is closer to the root.
So, 
	\[
	\Des(T)=\{(i,j) \, : \, i\lessdot_T j \mbox{ and } i>_\N j\}
	\]
and $\des(T)=|\Des(T)|$.
Given $x\in T$, we say that the \emph{depth} of $x$, denoted $\dpt(x)$, is the length of the unique path from $x$ to the root.
The \emph{depth} of $T$ is $\depth(T)\coloneqq \max_{x\in T} \dpt(x)$.
The \emph{major index} of $T$ is
	\[
	\maj(T)\coloneqq \sum_{(i,j)\in \Des(T)} \left(\depth(T)-\dpt(j)\right)
	\] 
\begin{remark}
Note that for any $x\in T$, the quantity $\depth(T)-\dpt(x)$ is precisely $\rho(x)$ where $\rho$ is the minimal rank function on the poset representation of $T$.
\end{remark}

\begin{proposition}[{\cite[Cor. 8.4]{PostnikovReinerWilliams}}]
For any connected building set $\mathcal{B}$ on $[n]$, the $h$-polynomial of the generalized permutohedron $P_{\mathcal{B}}$ is 
	\[
	h_{{\mathcal{B}}}(t)=\sum_{T}t^{\des(T)}
	\]
where the sum is over $\mathcal{B}$-trees $T$.	
\end{proposition}

Given  connected building sets $\Bc_1, \ldots, \Bc_r$ on pairwise disjoint sets $S_1,\ldots, S_r$, we can form the \Def{combined connected building set} $\Bc$ on $S=\bigcup_{i=1}^r S_i$ by $\Bc=\left( \bigsqcup_{i=1}^r\Bc_i\right)\sqcup \{S\}$.
We will now give a formula for the $h$-polynomial of such a building set.

\begin{proposition}
\label{prop:hpolyproduct}
Let $\Bc_1, \ldots, \Bc_r$ be connected building sets on the pairwise disjoint sets $S_1,\ldots, S_r$, and let $\Bc$ be the combined connected building set on $S=\bigcup_{i=1}^r S_i$.
Then 
	\[
	h_\Bc(t)=(1+t+\cdots+t^{r-1}) \prod_{i=1}^r h_{\Bc_i}(t).
	\]
\end{proposition}
\begin{proof}
Without loss of generality, let $S=[n]$ and let the sets $S_1,\ldots,S_r$ partition $[n]$ such that if $x\in S_i$ and $y\in S_j$, $x<y$ if and only if $i<j$ for every $1\leq i,j\leq r$.
Let $T$ be a $\mathcal{B}$-tree with vertex $i$ as the root.
Suppose that $i\in S_j$ for some $j$.
By Proposition~\ref{prop:BTreeConstruction}, $T$ is formed by connecting the root $i$ to the roots of trees on the connected components of $\mathcal{B}|_{[n]\setminus\{i\}}$.
Note that the connected components are precisesly $\mathcal{B}_k$ where $k\neq j$ and the connected components of $\mathcal{B}_j|_{S_j\setminus\{i\}}$. 
Therefore, $T$  is formed by $\mathcal{B}_k$-trees $T_1,T_2,\ldots,T_r$ such that for all $k\neq j$, the root of $T_k$ is connected to the root of $T_j$ for some $j=1,2,\ldots,r$.
Additionally, it is clear that for any collection of $\Bc_k$-trees, we can form a $\Bc$-tree by simply choosing one of the trees $T_j$ to be have the root.
Therefore, we will consider $T$ as being partitioned into $\Bc_k$-trees $T_1,T_2,\ldots,T_r$ with root in $T_j$ in this way.
Now, it is clear that $\des(T)=r-j+\sum_{k=1}^r \des(T_k)$ as the construction preserves all existing descents in each tree  $T_k$ and introduces exactly one new descent between $T_j$ and $T_k$ where $k>j$.
Since we the choices of trees for each $k$ are independent, the contribution of all trees where $T_j$ has the root to the $h$-polynomial is $t^{r-j}\prod_{k=1}^rh_{\Bc_k}(t)$.
Thus, summing over all choices of $j$ gives us the desired expression.
\end{proof}

Now we give a different characterization of the $q$-$h$-polynomial of the generalized permutohedron. This description comes from specializing Definition~\ref{d:bivariatehpolynomial} to the case of nestohedra, making use of alternative descriptions of the descent set and major index.

\begin{proposition}
\label{prop:nesto}
For any connected building set $\mathcal{B}$ on $[n]$, the $q$-$h$-polynomial of the generalized permutohedron $P_{\mathcal{B}}$ is 
	\[
	h_{{\mathcal{B}}}(t,q)=\sum_{T}t^{\des(T)}q^{\maj(T)}
	\]
where the sum is over $\mathcal{B}$-trees $T$.	
\end{proposition}

Define the statistic $\mu(T)\coloneqq\sum_{(i,j)\in T}\left( \depth(T)-\dpt(j) \right)$. 
Note that this statistic depends only on the isomorphism type of the rooted tree $T$ not on the labeling.
With this, we introduce a trivariate analogue of the $h$-polynomial of a nestohedron on connected building set
	\[
	h_\Bc(t,q,u)\coloneqq\sum_{T}t^{\des(T)}q^{\maj(T)}u^{\mu(T)}
	\]

By the Dehn-Sommerville relations, we have that the $h$-polynomial is palindromic.
In certain cases, we can provide a multivariate analogue of palindromicity.

\begin{theorem}
\label{thm:InvolutionPalindromic}
Let $\Bc$ be a connected building set on $[n]$ which is invariant under the involution $\omega:[n]\to[n]$ such that $\omega(i)=n-i+1$.
Then the $h$-polynomial for the nestohedron $P_\Bc$ is
	\[
	h_{\Bc}(t,q,u)=t^{n-1}h_\Bc(t^{-1},q^{-1},qu)
	\]
\end{theorem}
\begin{proof}
Let $\Bc$ be a building set such that $\omega(\Bc)=\Bc$.
Suppose that $T$ is a $\Bc$-tree. 
By Proposition \ref{prop:BTreeConstruction}, there exists a $\Bc$-tree $\tilde{T}$ such that $T$ and $\tilde{T}$ such that $\tilde{T}=\omega(T)$. 
That is, the trees are isomorphic as unlabeled rooted trees, and one can obtain the appropriate labels of one tree by applying the involution. 
It is clear that $\Des(\tilde{T})=\{(i,j) \, : \, (i,j)\not\in \Des(T)\}$. 
Hence $\des(\tilde{T})=n-1-\des(T)$ and $\maj(\tilde{T})=\mu(T)-\maj(T)$.
This gives the equality above.
\end{proof}

\section{Examples}
\label{sec:Examples}
We conclude with a section computing explicit examples of $q$-$h$-polynomials for nestohedra of interest. 
Included in the list are $S_n$-invariant nestohedra, graph associahedra, the associahedron, the stellahedron, and the Stanley--Pitman polytope.

\subsection{$S_n$-invariant nestohedra}
We will now specialize to the case of building sets which are invariant under the action of $S_n$ on the ground set $[n]$. 
Note that a connected building set $\Bc$ on $[n]$ is $S_n$-invariant if and only if 
	\[
	\Bc=\left\{\{1\},\ldots,\{n\},  {{[n]}\choose j}, \ j=k,\ldots,n \right\}
	\]
for some $2\leq k\leq n$.
Therefore, for a fixed $n$ and fixed $2\leq k\leq n$, we will denote this building set $\Bc_n^k$.

\begin{proposition}
\label{prop:invartrees}
Let $\Bc_n^k$ be the $S_n$-invariant connected building set of $[n]$ with minimal nonsingleton set of cardinality $k$.
Suppose that $T_1$ and $T_2$ are any two $\Bc$-trees.
Then $T_1$ and $T_2$ are isomorphic as unlabeled rooted trees. 
Moreover, for any $\Bc$-tree $T$, $T\cong A_{k-1}\oplus C_{n-k+1}$ as a poset, where $A_i$ is an antichain on $i$ elements, $C_j$ is a totally ordered chain on $j$ elements, and $\oplus$ is ordinal sum. 
\end{proposition}
\begin{proof}
This follows from Proposition \ref{prop:BTreeConstruction} with the observation that $\Bc_n^k|_{[n]\setminus \{i\}}\cong \Bc_{n-1}^{k-1}$ which is a connected building set. 
Continuing in this fashion, repeated restictions will result in connected building sets until we arrive at $\Bc_{n}^k|_{[n]\setminus W}$ where $W\subset [n]$ with $|W|=n-k+1$, which consists only of singleton elements.
\end{proof}

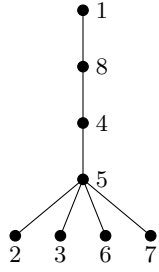
\begin{figure}
	\begin{tikzpicture}
		\draw[fill=black] (0,2) circle(2pt);
		\draw[fill=black] (0,1.25) circle(2pt);
		\draw[fill=black] (0,.5) circle(2pt);
		\draw[fill=black] (0,-.25) circle(2pt);
		\draw[fill=black] (-.9,-1) circle(2pt);
		\draw[fill=black] (-.3,-1) circle(2pt);
		\draw[fill=black] (.3,-1) circle(2pt);
		\draw[fill=black] (.9,-1) circle(2pt);
		
		\draw[black] (0,2)--(0,-.25);
		\draw[black] (0,-.25)--(-.9,-1);
		\draw[black] (0,-.25)--(.9,-1);
		\draw[black] (0,-.25)--(-.3,-1);
		\draw[black] (0,-.25)--(.3,-1);
		
		\node at (.25,2) {\footnotesize{$1$}};
		\node at (.25,1.25) {\footnotesize{$8$}};
		\node at (.25, .5) {\footnotesize{$4$}};
		\node at (.25,-.25) {\footnotesize{$5$}};
		\node at (-.9,-1.25) {\footnotesize{$2$}};
		\node at (-.3,-1.25) {\footnotesize{$3$}};
		\node at (.3,-1.25) {\footnotesize{$6$}};
		\node at(.9,-1.25) {\footnotesize{$7$}};
	\end{tikzpicture}
\caption{An example of a $\Bc_{8}^5$-tree $T$ with $\des(T)=4$ and $\maj(T)=8$.}
\label{fig:B58}
\end{figure}

\begin{theorem}
Let $\Bc_n^k$ be the $S_n$-invariant connected building set on $[n]$ with minimal nonsingleton set of cardinality $k$.
The $q$-$h$-polynomial for the nestohedron $P_{\Bc_n^k}$ is 
	\[
	h_{\Bc_n^k}(t,q)=\sum_{A\in {{[n]}\choose {n-k+1}}}\sum_{\pi\in S_A}t^{\des(\pi)+|\{ j\in [n]\setminus A \, : \, j>\pi_1\}|}q^{\maj(\pi)+\des(\pi)+|\{ j\in [n]\setminus A \, : \, j>\pi_1\}|}
	\]
Moreover, this polynomial satisfies 
	\[
	h_{\Bc_n^k}(t,q)=t^{n-1}q^{\frac{k^2-2kn-k+n^2+3n-2}{2}}h_{\Bc_n^k}(t^{-1},q^{-1}).
	\]
\end{theorem}

Prior to giving the proof of this formula, it is instructive to give concrete example of enumerating the descents in $\Bc_n^k$-trees. 
Consider the $\Bc_8^5$-tree $T$ given in Figure~\ref{fig:B58}. 
The descents which occur along the chain are precisely the descents of the permutation $\pi=5481\in S_{\{1,4,5,8\}}$ which has $\Des(5418)=\{1,3\}$ and $\des(5418)=2$. 
Moreover, there are descents which occur between the antichain and the chain itself. 
The number of such descents is precisely the number of elements of $[8]\setminus\{1,4,5,8\}$ which are larger than $5$. There are precisely $2$, and hence yielding $\des(T)=\des(5481)+|\{j\in [8]\setminus\{1,4,5,8\} \ : \ j>5\}|=4$.
When computing the major index, we note that the contributions of $\pi=5418$ is $\sum_{i\in \Des(5418)}(i+1)=\maj(5418)+\des(5418)=4+2=6$, to account for the correct rank.
Moreover, every descent between the antichain and the chain has rank $1$, so this contributes a total of $2$. 
Thus, $\maj(T)=\maj(5418)+\des(5418)+|\{j\in [8]\setminus\{1,4,5,8\} \ : \ j>5\}|=8$.

\begin{proof}
By Proposition \ref{prop:invartrees}, we know that any $T$ has the poset structure of $A_{k-1}\oplus C_{n-k+1}$.
So any labeled tree is described by an $n-k+1$-element subset $A$ of $[n]$ and a permutation $\pi\in S_A$. The permutation labels $C_{n-k+1}$, and the remaining elements of $[n]\setminus A$ label the antichain $A_{k-1}$.
There are two types of descents in the labeling: descents in $C_{n-k+1}$ which are enumerated by $\des(\pi)$, and descents where a label on the antichain $A_{k-1}$ is greater than $\pi_1$ which is enumerated by $|\{ j\in [n]\setminus A \, : \, j>\pi_1\}|$.
To compute $\maj(T)$, note that if $i\in \Des(\pi)$ this corresponds to $(j,\ell)\in \Des(T)$ such that $\rho(\ell)=i+1$.
 So the contribution from descents of this form is $q^{\maj(\pi)+\des(\pi)}$. 
The other descents are of the form $(i,\pi_1)\in\Des(T)$ and since $\rho(\pi_1)=1$, this contributes $q^{|\{ j\in [n]\setminus A \, : \, j>\pi_1\}|}$.

To see the palidromicity statement, note that since $\Bc_n^k$ is $S_n$-invariant, then it is invariant under the involution $\omega(i)=n-i+1$.
It is clear that $\mu(T)=k-2+\sum_{i=1}^{n-k+1}i=\frac{k^2-2kn-k+n^2+3n-2}{2}$ for any $\Bc_n^k$-tree $T$.
Subsequently, applying the result of  Theorem \ref{thm:InvolutionPalindromic} and setting $u=1$ yields the desired statement. 
\end{proof}

\subsection{Graph associahedra}
We now consider a large family of examples of nestohedra arising from graphs.  
Given a graph $G=([n],E)$, a \Def{tube} of $G$ is a proper, nonempty subset $I\subset [n]$ such that the induced subgraph $G|_I$ is connected. 
A \Def{$k$-tubing} of $G$, $\chi$, is a a collection of $k$ distinct tubes subject to:
	\begin{enumerate}
	\item For all incomparable $A_1,A_2\in \chi$, $A_1\cup A_2\not\in \chi$ (\Def{non-adjacency});
	\item For all incomparable $A_1,A_2\in \chi$, $ A_1\cap A_2=\emptyset$ (\Def{non-intersecting}).
	\end{enumerate}	
We do, however, allow for $A_1\subset A_2$, which is called a  \Def{nesting}.
We say that a tubing $\chi$ is \Def{maximal} if it cannot add any additional tubes to $\chi$, or equivalently, if $|\chi|=n-1$.  
Given a graph $G$, the \Def{graph associahedron of $G$} is the polytope $P_G$ whose face lattice is given by the set of all tubings of $G$ where $\chi<\chi'$ if $\chi$ is obtained from $\chi'$ by adding tubes. 
Subsequently, the vertices of $P_G$ correspond to maximal tubings. 
This notion of graph associahedra originates with Carr and Devadoss \cite{CarrDevadoss,Devadoss-Realization} and has been a well-studied family of examples of simple generalized permutohedra (see, e.g., \cite{ArdilaReinerWilliams,BarnardMcConville,CardinalLangermanPerez-Lantero,CarrDevadossForcey,MannevillePilaud}).

\begin{remark}
\label{rmk:graphconnected}
Given a simple graph $G=([n],E)$, the graph associahedron $P_G$ is an example of nestohedron on a connected building set, even when $G$ is not a connected graph.
The \emph{graphical building set} of $G$, $\mathcal{B}(G)$ is the collection of nonempty $J\subseteq[n]$ such that the induced subgraph $G|_J$ is connected.
While the building set $B(G)$ is connected if and only if $G$ is connected (c.f. \cite[Ex. 6.2]{PostnikovReinerWilliams}), the graph associahedra $P_G$ using the notions of  Carr and Devadoss \cite{CarrDevadoss,Devadoss-Realization} is the nestahedron with building set $\widehat{\mathcal{B}(G)}=\mathcal{B}(G)\cup[n]$ which is always connected and $\widehat{\mathcal{B}(G)}=\mathcal{B}(G)$ if $G$ connected.  
\end{remark} 

In light of Remark \ref{rmk:graphconnected}, we can specialize Proposition \ref{prop:hpolyproduct} to determine the $h$-polynomial of a disconnected graph.  

\begin{corollary}
Let $G$ be a simple graph on $[n]$ with connected components $G_1,G_2,\ldots,G_k$. 
Then 
	\[
	h_{G}(t)=(1+t+\cdots+t^{k-1})\prod_{i=1}^kh_{G_i}(t).
	\]
\end{corollary}

Let $G=([n],E)$ be a simple graph and let $\chi$ be a maximal tubing of $G$. 
Given $i\in [n]$, the \Def{nesting index} of $i$, denoted $\nu_\chi(i)$, is the number of tubes containing $i$. 
The \Def{nesting number} of $\chi$ is $\operatorname{nest}(\chi)\coloneqq \max_{i\in[n]}\nu_{\chi}(i)$.
Given any maximal $\chi$, observe that for any tube $A_j\in \chi$, there exists a unique element $\alpha_j\in A_j$ such that for any tube $A_k\subset A_j$, we have $\alpha_j\not\in A_k$.
For convenience, we will write $A_k\lessdot A_j$ if $A_k\subset A_j$ and there is no tube $A_\ell$ such that $A_k\subset A_\ell\subset A_\ell$.
Let $\alpha_n$ denote the unique element which is not contained in any tube of $\chi$. 

The \Def{nesting descent set} is
	\[
	\operatorname{NestDes}(\chi)\coloneqq \{(\alpha_k,\alpha_j) \, : \, \alpha_k>\alpha_j \mbox{ and } A_k\lessdot A_j \}\cup \{(\alpha_\ell,\alpha_n) \, : \, \alpha_\ell>\alpha_n \mbox{ and } A_\ell\not\subset A_p \mbox{ for any } A_p\}.
	\]	
The \Def{nesting descent number} is 
	\[
	\operatorname{nestDes}(\chi)\coloneqq |\operatorname{NestDes}(\chi)|
	\]
and the \Def{nesting major index} is
	\[
	\operatorname{nestMaj}(\chi)\coloneqq \sum_{(\alpha_k,\alpha_j)\in \operatorname{NestDes}(\chi)} \left( \operatorname{nest}(\chi)-\nu_{\chi}(\alpha_j)\right)
	\]

We now state a formula for the $q$-$h$-polynomial of graph associahedra in terms of graph tubings.

\begin{proposition}
Let $G$ be a simple graph. 
The $q$-$h$-polynomial is
	\[
	h_G(t,q)=\sum_\chi t^{\operatorname{nestDes}(\chi)} q^{\operatorname{nestMaj}(\chi)}
	\]
where the sum is taken over all maximal tubings $\chi$.	
\end{proposition}
\begin{proof}
This follows by unpacking the definitions of $\Bc$-trees in terms of graph tubings and applying Proposition \ref{prop:nesto}. 
\end{proof}

\begin{remark}
As was the case with nestohedra in general, we should note that this polynomial is invariant only under labeled graph automorphisms. 
Under most circumstance, a different choice of labeling of the vertices $G$ will produce a different bivariate polynomial. 
However, the specialization under $q=1$ is invariant under permutation of the ground set. 
\end{remark}

\begin{remark}
As with nestohedra, we can similarly define a trivariant polynomial for graph associahedra, namely
	\[
	h_G(t,q,u)=\sum_\chi t^{\operatorname{nestDes}(\chi)} q^{\operatorname{nestMaj}(\chi)}u^{\mu(\chi)}
	\]
where the sum ranges over all maximal and $\mu(\chi)=\sum_{(\alpha_k,\alpha_j)}(\operatorname{nest}(T)-\nu_T(\alpha_j))$ where this sum is over all pairs $(\alpha_k,\alpha_j)$ such that $A_k\lessdot A_j$, which is a direct translation of the $\mu$ statistic for nestohedra.
If the involution $\omega:[n]\to[n]$ such that $\omega(i)=n-i+1$ produces a labeled graph automorphism, then Theorem \ref{thm:InvolutionPalindromic} gives  us that palindromicity statement
	\[
	h_G(t,q,u)=t^{n-1}h_G(t^{-1},q^{-1},qu).
	\]
There are only two $S_n$ invariant graphs, namely the complete graph $K_n$ and the null graph $N_n=\overline{K_n}$ (i.e. the edgeless graph), which produce only the simplest examples of generalized permutohedra.
$P_{K_n}$ is the usual permutohedron $\Pi_n$, and hence $h_{K_n}(t,q)$ is the usual Euler--Mahonian polynomial.
$P_{N_n}$ is simply an $n-1$ dimensional simplex and thus $h_{N_n}(t,q)=\sum_{i=0}^{n-1}(tq)^i$. 	
\end{remark}

\subsection{The associahedron and a new $q$-analogue of Narayana numbers}
\label{sec:associahedron}
\newcommand{\Path}{\mathsf{Path}}%
The associahedron $\mathsf{A}(n)$, which first appeared in the work of \cite{Stasheff}, as well as the notable work of Lee \cite{Lee-Associahedron}, is the graph associahedron for $G=\Path(n)$, where the vertices are labeled linearly.
It is well-known that 
	\[
	h_{\Path(n)}(t)=\sum_{k=1}^{n}N(n,k) t^{k-1}
	\]
where $N(n,k)=\frac{1}{n}{n\choose k}{n\choose{k-1}}$ is the \Def{Narayana number}, which refine the Catalan numbers. 
That is, $h_{\Path(n)}(1)=C_n$.
To verify this formula, one should note that $\Bc$-trees, or graph tubings on $\Path(n)$, are in bijection with binary trees on $n$ vertices (See {\cite[Sec. 8.2]{Postnikov-Beyond}}).
The bijection sends descents in a $\Bc$-tree to right edges in an unlabeled binary tree and $N(n,k)$ is known to enumerate the number of unlabeled binary trees on $n$ vertices with $k-1$ right edges.
Subsequently, we will phrase all formulae in terms of binary trees. 

Let $T$ be a binary tree. 
Given an edge $e\in T$, let $\dpt(e)$ be the length of the path from the root vertex to the closest vertex incident with $e$.
Let $\depth(T)=\max_{e\in T} \dpt(e)$.
The \Def{right multiset} of $T$ is the multiset 
	\[
	\Rc(T)\coloneqq\left\{\dpt(e) \ : \ e \mbox{ is a right edge of } T \right\}.
	\]	
The \Def{right number} of $T$ is $r(T)=|\Rc(T)|$ and the \Def{right index} of $T$ is 
	\[
	\operatorname{rindex}(T)\coloneqq \depth(T)r(T)-\sum_{j\in \Rc(T)}j.
	\]	
By translating the general results for nestohedra into the above language for binary trees, we have the following:
	\begin{corollary}
	The $q$-$h$-polynomial for the associahedron is 
		\[
		h_{\Path(n)}(t,q)=\sum_{T} t^{r(T)}q^{\operatorname{rindex}(T)}
		\]
	where the sum ranges over all rooted unlabelled binary tree $T$ on $n$ vertices.
	\end{corollary}	
	
\begin{remark}
This theorem gives rise to \emph{a} $q$-analogue of the Narayana numbers.
We say the \Def{(alternative) $q$-Narayana number} is 
	\[
	N(n,k,q)=\sum_{\substack{T\\r(T)=k-1}}q^{\operatorname{rindex}(T)}.
	\]
It is clear that the substitution $q=1$ yields $N(n,k)$ as desired. 
We call these the \emph{alternative} $q$-Narayana numbers because, while this is the natural $q$-analogue in the context of generalized permutohedra as it arises from the major index, this does not agree with the usual $q$-Narayana number in the literature (see, e.g., \cite{MR2047757, ReinerSommer-qKerewera} ).
\end{remark}	

\subsection{The stellahedron}
The \Def{star graph} on $n+1$ vertices is the complete bipartite graph $K_{1,n}$.
The \Def{stellohedron} is the graph associahedron associated to $K_{1,n}$. 
Let $K_{1,n}$ be labeled such that the center vertex is labeled $n+1$.
The $\Bc$-trees for $K_{1,n}$ are in bijection with partial permutations of $[n]$.
In particular, the structure of a $\Bc$-tree is given by the ordinal sum of an antichain with a totally ordered chain $A_{n-k-1}\oplus C_{k+1}$ for some $k=0,\ldots,n$ such that the minimal element of $C_{k+1}$ has label $n+1$.

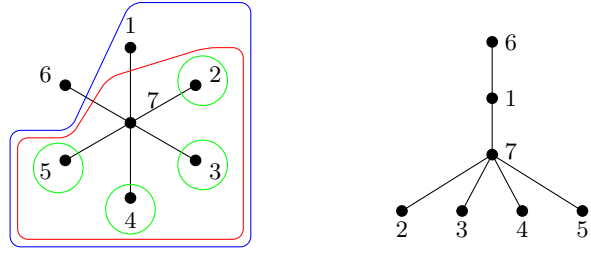
\begin{figure}
	\begin{tikzpicture}
	\draw[fill=black] (0,0) circle(2pt);
	\draw[fill=black] (0,1) circle(2pt);
	\draw[fill=black] (0,-1) circle(2pt);
	\draw[fill=black] (.866,.5) circle(2pt);
	\draw[fill=black] (.866,-.5) circle(2pt);
	\draw[fill=black] (-.866,-.5) circle(2pt);
	\draw[fill=black] (-.866,.5) circle(2pt);
	
	\draw[black] (0,1)--(0,-1);
	\draw[black] (.866,.5)--(-.866,-.5);
	\draw[black] (-.866,.5)--(.866,-.5);

	\node at (1.126,.65) {\footnotesize{$2$}};	
	\node at (1.126,-.65) {\footnotesize{$3$}};	
	\node at (-1.126,-.65) {\footnotesize{$5$}};	
	\node at (-1.126,.65) {\footnotesize{$6$}};	
	\node at (0,1.3) {\footnotesize{$1$}};
	\node at (0,-1.3) {\footnotesize{$4$}};
	\node at (.3,.3) {\footnotesize{$7$}};
	
	\draw[blue,rounded corners] (0,1.6)--(1.6,1.6)--(1.6,-1.65)--(-1.6,-1.65)--(-1.6,-.1)--(-.8,-.1)--cycle;
	
	\draw[red,rounded corners] (-.8,-.2)--(-.3,.6)--(1,1)--(1.5,1)--(1.5,-1.55)--(-1.5,-1.55)--(-1.5,-.2)--cycle;	
	
	\draw[green] (.96,.56) circle[radius=.33];
	\draw[green] (.96,-.56) circle[radius=.33];
	\draw[green] (-.96,-.6) circle[radius=.33];
	\draw[green] (0,-1.15) circle[radius=.33];	
	\end{tikzpicture}
	\hspace{1.6cm}
	\begin{tikzpicture}
	\draw[fill=black] (0,1.75) circle(2pt);
	\draw[fill=black] (0,1) circle(2pt);
	\draw[fill=black] (0,.25) circle(2pt);
	\draw[fill=black] (-1.2,-.5) circle(2pt);
	\draw[fill=black] (-.4,-.5) circle(2pt);
	\draw[fill=black] (.4,-.5) circle(2pt);
	\draw[fill=black] (1.2,-.5) circle(2pt);
	
	\draw[black] (0,1.75)--(0,.25);
	\draw[black] (0,.25)--(-1.2,-.5);
	\draw[black] (0,.25)--(-.4,-.5);
	\draw[black] (0,.25)--(.4,-.5);
	\draw[black] (0,.25)--(1.2,-.5);
	
	\node at (-1.2,-.75) {\footnotesize{$2$}};
	\node at (-.4,-.75) {\footnotesize{$3$}};
	\node at (.4,-.75) {\footnotesize{$4$}};
	\node at (1.2,-.75) {\footnotesize{$5$}};
	\node at (.25,1.75) {\footnotesize{$6$}};
	\node at (.25,1) {\footnotesize{$1$}};
	\node at (.25,.3) {\footnotesize{$7$}};	
	\end{tikzpicture}
\caption{A tubing of $K_{1,6}$ and its corresponding $\Bc$-tree.}
\label{fig:star}	
\end{figure}

To see this, note that we can identify the $\Bc$-trees with graph tubings. 
Any tubing of $K_{1,n}$ is either (i) the tubing where each vertex $i=1,2,\ldots, n$ is in a singleton tube and $n+1$ is the root, or (ii) some vertex $i$ is the root and we have a tube containing all other vertices.
In the case of (ii), once  $i$ is chosen, then the tubing directly arises from a tubing of $K_{1,n-1}$ on the labels $[n+1]\setminus \{i\}$. 
Thus, by induction, we will have $\Bc$-trees of the proposed form.
For example, consider the tubing and $\Bc$-tree given in Figure~\ref{fig:star}, which corresponds to the partial permutation $\pi=61$ on $[6]$.

Subsequently, the elements of the $C_{k+1}$ above the $n+1$ are the partial permutation (see {\cite[Sec. 10.4]{PostnikovReinerWilliams}})
With this in mind, we can state the $q$-analogue of the $h$-polynomial for the stellohedron.

\begin{proposition}
The $q$-$h$-polynomial for the stellohedron is 
	\[
	h_{K_{1,n}}(t,q)=1+\sum_{w}t^{\des(w)+1}q^{\maj(w)+2\des(w)+2}
	\]
where the sum is over all nonempty partial permutations of $[n]$.	
\end{proposition}
\begin{proof}
The labels on $C_{k+1}$ correspond to a partial permutation of $\tilde{w}$ of $[n+1]$ where $\tilde{w}_1=n+1$.
Thus, we consider $w$ to be the partial permutation of $[n]$ with this first element omitted. 
If $w=\emptyset$, the corresponding $\Bc$-tree has no descents.
If $w\neq \emptyset$, then the corresponding $\Bc$-tree $T$ has precisely $\des(w)+1$ descents, due to the guaranteed descent between $n+1$ and $w_1$.
When computing the major index, note that if $i\in \Des(w)$, this means that we have an element of rank $i+2$ where a descent occurs in $T$.
Hence, the contribution to the major index is $\sum_{i\in\Des(w)}(i+2)=\maj(w)+2\des(w)$. 
Additionally, the descent between $n+1$ and $w_1$ contributes $2$, as $\rho(w_1)=2$. 
Thus, we have the desired formula.  
\end{proof}

\subsection{The Stanley--Pitman polytope}
Introduced by Stanley and Pitman in \cite{StanleyPitman}, the \Def{Stanley--Pittman polytope} is a integral polytope defined by the equations
	\[
	\mathsf{PS}(n)\coloneqq \left\{\x\in\R^n \, : \, x_i\geq 0 \mbox{ and } \sum_{i=1}^jx_i \leq j \mbox{ for each } 1\leq j \leq n\right\}.
	\]
This polytope is combinatorially equivalent to an $n$-cube, as illustrated in Figure~\ref{fig:PS}.
However, this polytope is of particular interest  as it appears naturally when studying empirical distributions in statistics and has connections to many combinatorial objects, such as parking functions and plane trees. 
Postnikov {\cite[Sec. 8.5]{Postnikov-Beyond}} observed that this polytope can be realized as the nestohedron from the building set 
\newcommand{\BPS}{\mathcal{B}_{\mathsf{PS}}}%
	\[
	\BPS=\{ [i,n], \ \{i\} \ : \ i\in [n] \},
	\]
where $[i,n]=\{i,i+1,\ldots,n\}$. 
Notably, this is not a graph associahedron. 
Given that this polytope is combinatorially equivalent to an $n$-cube, we have $h_{\BPS}(t)=(1+t)^{n-1}$  {\cite[Thm. 20]{StanleyPitman}}.
We now give the $q$-analogue.

 \begin{figure}
	\centering
\begin{tikzpicture}
	\draw[black,thick,->] (0,0)--(2,0) node[anchor=west]{$x_1$};
	\draw[black,thick,->] (0,0)--(0,3) node[anchor=west]{$x_2$};
	
	\draw[fill=red] (0,0) circle(2pt);
	\draw[fill=red] (1,0) circle(2pt);
	\draw[fill=red] (1,1) circle(2pt);
	\draw[fill=red] (0,2) circle(2pt);
	
	\draw[fill=gray] (0,1) circle(2pt);
	
	\draw[red,thick] (0,0)--(1,0)--(1,1)--(0,2)--cycle;
	
	\draw[fill=red,opacity=0.3] (0,0)--(1,0)--(1,1)--(0,2)--cycle;

\end{tikzpicture}
\hspace{1.5cm}	
	\tdplotsetmaincoords{60}{130}
\begin{tikzpicture}
[tdplot_main_coords,
			grid/.style={very thin,gray},
			axis/.style={->,black,thick}]

	\draw[axis] (0,0,0) -- (3.5,0,0) node[anchor=west]{$x_1$};
	\draw[axis] (0,0,0) -- (0,3.5,0) node[anchor=west]{$x_2$};
	\draw[axis] (0,0,0) -- (0,0,3.5) node[anchor=west]{$x_3$};
    
    \draw[fill=blue] (0,0,0) circle(2pt);
    \draw[fill=blue] (1,0,0) circle(2pt);
    \draw[fill=blue] (1,1,0) circle(2pt);
    \draw[fill=blue] (0,2,0) circle(2pt);
    \draw[fill=blue] (1,1,1) circle(2pt);
    \draw[fill=blue] (1,0,2) circle(2pt);
    \draw[fill=blue] (0,2,1) circle(2pt);
    \draw[fill=blue] (0,0,3) circle(2pt);
    
	\draw[fill=gray] (0,0,1) circle(2pt);    
    \draw[fill=gray] (0,0,2) circle(2pt);
    \draw[fill=gray] (0,1,0) circle(2pt);
    \draw[fill=gray] (0,1,1) circle(2pt);
    \draw[fill=gray] (0,1,2) circle(2pt);
    \draw[fill=gray] (1,0,1) circle(2pt);
    
    \draw[blue,thick] (0,0,0)--(1,0,0)--(1,1,0)--(0,2,0)--cycle;
    \draw[blue,thick] (0,0,3)--(1,0,2)--(1,1,1)--(0,2,1)--cycle;
    \draw[blue,thick] (0,0,0)--(0,0,3);
	\draw[blue,thick] (1,0,0)--(1,0,2);
	\draw[blue,thick] (1,1,0)--(1,1,1);
	\draw[blue,thick] (0,2,0)--(0,2,1);

    \draw[fill=blue,opacity=0.3](0,0,0)--(1,0,0)--(1,1,0)--(0,2,0)--cycle;
    \draw[fill=blue,opacity=0.3](0,0,3)--(1,0,2)--(1,1,1)--(0,2,1)--cycle;
    \draw[fill=blue,opacity=0.3](0,0,0)--(1,0,0)--(1,0,2)--(0,0,3)--cycle;
    \draw[fill=blue,opacity=0.3](0,0,0)--(0,2,0)--(0,2,1)--(0,0,3)--cycle;
    \draw[fill=blue,opacity=0.3](1,0,0)--(1,1,0)--(1,1,1)--(1,0,2)--cycle;
    \draw[fill=blue,opacity=0.3](1,1,0)--(1,1,1)--(0,2,1)--(0,2,0)--cycle;
    
\end{tikzpicture}
\caption{$\mathsf{PS}(2)$ and $\mathsf{PS}(3)$}
\label{fig:PS}
\end{figure}

	\begin{proposition}
	The $q$-$h$-polynomial for the Stanley--Pitman polytope is 
		\[
		h_{\BPS}(t,q)=\sum_{\ell=0}^{n-2} {{n-2} \choose \ell} t^\ell q^{\frac{\ell^2+3\ell+2}{2}}\left( t+q^\ell\right).
		\]
	\end{proposition}

\begin{proof}
First note that $h_{\BPS}(t,1)=(t+1)^{n-1}$, so this agrees with the known results. To compute this, we will need $\BPS$-trees, which as  determined by Postnikov, Reiner, and Williams {\cite[Sec. 10.5]{PostnikovReinerWilliams}},
are formed in the following way. 
Given any increasing sequence of positive integers $I=\{i_1<i_2<\cdots<i_k=n\}$ where we let $i_1$ be the root and form the chain of edges $(i_1,i_2),(i_2,i_3),\ldots, (i_{k-1},i_k)$ and for all $j\in [n]\setminus I$ we have the edge $(i_s,j)$ where $i_s$ is the minimal element of $I$ such that $i_s>j$.
An example can be seen in Figure~\ref{fig:PSTrees}.

It is clear that all descents will be occur along the chain of edges. So, we must consider two cases: (i) $i_{k-1}=n-1$ and (ii) $i_{k-1}\leq n-2$.
In case (i), for convenience let $\ell=k-2$. 
We form a tree $T$ by choosing a subset $J\in {{[n-2]}\choose \ell}$ and arranging it increasing order to form a chain of edges which ends in $(i_\ell,n-1),(n-1,n)$. 
By definition, $\depth(T)=\ell+1$, $\des(T)=\ell+1$, and $\maj(T)=(\ell+1)^2-\sum_{i=0}^\ell i = \frac{\ell^2+3\ell+2}{2}.$
So, the contribution of trees of this form the $q$-$h$-polynomials is
	\begin{equation}
	\label{eq:PStree1}
	\sum_{\ell=0}^{n-2} {{n-2}\choose \ell} t^{\ell+1} q^{\frac{\ell^2+3\ell+2}{2}}.
	\end{equation}

In case (ii) where $i_{k-1}\neq n-1$, for ease of notation, let $\ell=k-1$.
Similarly, we form such a tree $T$ by choosing $J\in {{[n-2]}\choose \ell}$ and arranging it increasing order to form a chain of edges which ends in $(i_\ell,n)$. Note that, when including the elements not in the chain, we gain edges from the vertex $n$ going away from the root, in particular, the edge $(n,n-1)$. 
So, we again have $\depth(T)=\ell+1$. 
However, we now have  $\des(T)=\ell$, and $\maj(T)=(\ell+1)^2-\sum_{i=0}^{\ell-1} i = \frac{\ell^2+5\ell+2}{2}.$
So the contribution of trees of this type to the $q$-$h$-polynomial is 
	\begin{equation}
	\label{eq:PStree2}
	\sum_{\ell=0}^{n-2} {{n-2}\choose \ell} t^{\ell} q^{\frac{\ell^2+5\ell+2}{2}}.
	\end{equation}

Summing (\ref{eq:PStree1}) and (\ref{eq:PStree2}) and simplifying gives the desired expression.	
\end{proof}

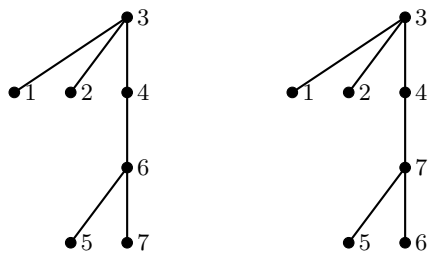
\begin{figure}
	\begin{tikzpicture}
	\draw[fill=black] (0,2) circle(2pt) node[anchor=west]{\footnotesize $3$};
	\draw[fill=black] (0,1) circle(2pt) node[anchor=west]{\footnotesize $4$};
	\draw[fill=black] (0,0) circle(2pt) node[anchor=west]{\footnotesize $6$};
	\draw[fill=black] (0,-1) circle(2pt) node[anchor=west]{\footnotesize $7$};

	\draw[fill=black] (-.75,-1) circle(2pt) node[anchor=west]{\footnotesize $5$};
	
	\draw[fill=black] (-.75,1) circle(2pt) node[anchor=west]{\footnotesize $2$};
	\draw[fill=black] (-1.5,1) circle(2pt) node[anchor=west]{\footnotesize $1$};	
	
	\draw[black,thick](0,2)--(0,-1);
	
	\draw[black,thick](0,0)--(-.75,-1);
	\draw[black,thick](0,2)--(-1.5,1);
	\draw[black,thick](0,2)--(-.75,1);
	\end{tikzpicture}
	\hspace{1.5cm}
	\begin{tikzpicture}
	\draw[fill=black] (0,2) circle(2pt) node[anchor=west]{\footnotesize $3$};
	\draw[fill=black] (0,1) circle(2pt) node[anchor=west]{\footnotesize $4$};
	\draw[fill=black] (0,0) circle(2pt) node[anchor=west]{\footnotesize $7$};
	\draw[fill=black] (0,-1) circle(2pt) node[anchor=west]{\footnotesize $6$};

	\draw[fill=black] (-.75,-1) circle(2pt) node[anchor=west]{\footnotesize $5$};
	
	\draw[fill=black] (-.75,1) circle(2pt) node[anchor=west]{\footnotesize $2$};
	\draw[fill=black] (-1.5,1) circle(2pt) node[anchor=west]{\footnotesize $1$};	
	
	\draw[black,thick](0,2)--(0,-1);
	
	\draw[black,thick](0,0)--(-.75,-1);
	\draw[black,thick](0,2)--(-1.5,1);
	\draw[black,thick](0,2)--(-.75,1);
	\end{tikzpicture}
\caption{Two $\BPS$-trees for $n=7$ from the increases sequences $I_1=\{3<4<6<7\}$ and $I_2=\{3<4<7\}$.
Alternatively, these are the two trees from the set $\{3,4\}\subset[5]$.}	
\label{fig:PSTrees}
\end{figure}

\begin{remark}
We conclude our discussion by noting that our computation produces an \emph{alternative $q$-analogue} of ${{n-1} \choose \ell}$, namely
	\[
	{{n-2}\choose {\ell-1}} q^{\frac{\ell^2+\ell}{2}} +{{n-2}\choose {\ell}} q^{\frac{\ell^2+5\ell+2}{2}}.
	\] 
This of course reduces to ${{n-1}\choose \ell}$ when $q=1$ and arises quite naturally from generalizing the major index statistic.
However, this is not the usual $q$-analogue of a binomial coefficient which arises in many natural ways, such as bit string inversions and lattice path areas. 	
\end{remark}

%
%


\bibliographystyle{plain}
\bibliography{MVGPBib}

\begin{thebibliography}{10}

\bibitem{AdinBrentiRoichman-Hyper}
Ron~M. Adin, Francesco Brenti, and Yuval Roichman.
\newblock Descent numbers and major indices for the hyperoctahedral group.
\newblock {\em Adv. in Appl. Math.}, 27(2-3):210--224, 2001.
\newblock Special issue in honor of Dominique Foata's 65th birthday
  (Philadelphia, PA, 2000).

\bibitem{AguiarArdila}
Marcelo Aguiar and Federico Ardila.
\newblock Hopf monoids and generalized permutohedra, 2017.
\newblock arXiv:1709.07504.

\bibitem{ArdilaReinerWilliams}
Federico Ardila, Victor Reiner, and Lauren Williams.
\newblock Bergman complexes, {C}oxeter arrangements, and graph associahedra.
\newblock {\em S\'{e}m. Lothar. Combin.}, 54A:Art. B54Aj, 25, 2005/07.

\bibitem{Armstrong}
Drew Armstrong.
\newblock Generalized noncrossing partitions and combinatorics of {C}oxeter
  groups.
\newblock {\em Mem. Amer. Math. Soc.}, 202(949):x+159, 2009.

\bibitem{BagnoBiagioli}
Eli Bagno and Riccardo Biagioli.
\newblock Colored-descent representations of complex reflection groups
  {$G(r,p,n)$}.
\newblock {\em Israel J. Math.}, 160:317--347, 2007.

\bibitem{BarnardMcConville}
Emily Barnard and Thomas Mc{C}onville.
\newblock Lattices from graph associahedra and subalgebras of the
  {M}alvenuto-{R}eutenauer algebra, 2018.
\newblock arXiv:1808.05670.

\bibitem{BeckBraun-EM}
Matthias Beck and Benjamin Braun.
\newblock Euler-{M}ahonian statistics via polyhedral geometry.
\newblock {\em Adv. Math.}, 244:925--954, 2013.

\bibitem{MR2047757}
Petter Br\"{a}nd\'{e}n.
\newblock {$q$}-{N}arayana numbers and the flag {$h$}-vector of {$J(\bold
  2\times\bold n)$}.
\newblock {\em Discrete Math.}, 281(1-3):67--81, 2004.

\bibitem{BraunOlsen-EulerMahonianSemigroup}
Benjamin Braun and McCabe Olsen.
\newblock Euler-{M}ahonian statistics and descent bases for semigroup algebras.
\newblock {\em European J. Combin.}, 69:237--254, 2018.

\bibitem{CardinalLangermanPerez-Lantero}
Jean Cardinal, Stefan Langerman, and Pablo P\'{e}rez-Lantero.
\newblock On the diameter of tree associahedra.
\newblock {\em Electron. J. Combin.}, 25(4):Paper 4.18, 13, 2018.

\bibitem{CarrDevadossForcey}
Michael Carr, Satyan~L. Devadoss, and Stefan Forcey.
\newblock Pseudograph associahedra.
\newblock {\em J. Combin. Theory Ser. A}, 118(7):2035--2055, 2011.

\bibitem{CarrDevadoss}
Michael~P. Carr and Satyan~L. Devadoss.
\newblock Coxeter complexes and graph-associahedra.
\newblock {\em Topology Appl.}, 153(12):2155--2168, 2006.

\bibitem{Devadoss-Realization}
Satyan~L. Devadoss.
\newblock A realization of graph associahedra.
\newblock {\em Discrete Math.}, 309(1):271--276, 2009.

\bibitem{Euler}
Leonhard Euler.
\newblock Remarques sur un beau rapport entre les series des puissances tant
  direct que reciproques.
\newblock {\em Mem. L'Acad. Sci. Berlin}, 17:83--106, 1768.

\bibitem{FeichtnerSturmfels}
Eva~Maria Feichtner and Bernd Sturmfels.
\newblock Matroid polytopes, nested sets and {B}ergman fans.
\newblock {\em Port. Math. (N.S.)}, 62(4):437--468, 2005.

\bibitem{FominReading}
Sergey Fomin and Nathan Reading.
\newblock Generalized cluster complexes and {C}oxeter combinatorics.
\newblock {\em Int. Math. Res. Not.}, (44):2709--2757, 2005.

\bibitem{Lee-Associahedron}
Carl~W. Lee.
\newblock The associahedron and triangulations of the {$n$}-gon.
\newblock {\em European J. Combin.}, 10(6):551--560, 1989.

\bibitem{MacMahon}
Percy~A. MacMahon.
\newblock {\em Combinatory analysis}.
\newblock Two volumes (bound as one). Chelsea Publishing Co., New York, 1960.

\bibitem{MannevillePilaud}
Thibault Manneville and Vincent Pilaud.
\newblock Graph properties of graph associahedra.
\newblock {\em S\'{e}m. Lothar. Combin.}, 73:Art. B73d, 31, [2014-2016].

\bibitem{PostnikovReinerWilliams}
Alex Postnikov, Victor Reiner, and Lauren Williams.
\newblock Faces of generalized permutohedra.
\newblock {\em Doc. Math.}, 13:207--273, 2008.

\bibitem{Postnikov-Beyond}
Alexander Postnikov.
\newblock Permutohedra, associahedra, and beyond.
\newblock {\em Int. Math. Res. Not. IMRN}, (6):1026--1106, 2009.

\bibitem{ReinerSommer-qKerewera}
Victor Reiner and Eric Sommers.
\newblock Weyl group {$q$}-{K}reweras numbers and cyclic sieving.
\newblock {\em Ann. Comb.}, 22(4):819--874, 2018.

\bibitem{StanleyPitman}
Richard~P. Stanley and Jim Pitman.
\newblock A polytope related to empirical distributions, plane trees, parking
  functions, and the associahedron.
\newblock {\em Discrete Comput. Geom.}, 27(4):603--634, 2002.

\bibitem{Stasheff}
James~Dillon Stasheff.
\newblock Homotopy associativity of {$H$}-spaces. {I}, {II}.
\newblock {\em Trans. Amer. Math. Soc. 108 (1963), 275-292; ibid.},
  108:293--312, 1963.

\bibitem{Ziegler-Book}
G\"{u}nter~M. Ziegler.
\newblock {\em Lectures on polytopes}, volume 152 of {\em Graduate Texts in
  Mathematics}.
\newblock Springer-Verlag, New York, 1995.

\end{thebibliography}

\end{document}